\documentclass[12pt]{amsart}

\usepackage{enumerate, amsmath, amsthm, amsfonts, amssymb, mathrsfs, graphicx, paralist, stmaryrd}
\usepackage[all]{xy}
\usepackage[usenames, dvipsnames]{color}
\usepackage[margin=1in]{geometry} 
\usepackage[bookmarks, bookmarksdepth=2, colorlinks=true, linkcolor=blue, citecolor=blue, urlcolor=blue]{hyperref}
\usepackage{eucal}

\newcommand{\bA}{\mathbf{A}}

\newcommand{\bC}{\mathbf{C}}

\newcommand{\bE}{\mathbf{E}}

\newcommand{\bG}{\mathbf{G}}

\newcommand{\bH}{\mathbf{H}}

\newcommand{\bL}{\mathbf{L}}

\newcommand{\bM}{\mathbf{M}}

\newcommand{\bN}{\mathbf{N}}

\newcommand{\bO}{\mathbf{O}}
\newcommand{\cO}{\mathcal{O}}

\newcommand{\bS}{\mathbf{S}}

\newcommand{\fS}{\mathfrak{S}}

\newcommand{\bU}{\mathbf{U}}

\newcommand{\bV}{\mathbf{V}}

\newcommand{\bW}{\mathbf{W}}

\newcommand{\bk}{\mathbf{k}}

\let\ol\overline
\let\ul\underline

\renewcommand{\phi}{\varphi}

\newcommand{\arxiv}[1]{\href{http://arxiv.org/abs/#1}{{\tt arXiv:#1}}}

\makeatletter
\def\Ddots{\mathinner{\mkern1mu\raise\p@
\vbox{\kern7\p@\hbox{.}}\mkern2mu
\raise4\p@\hbox{.}\mkern2mu\raise7\p@\hbox{.}\mkern1mu}}
\makeatother

\DeclareMathOperator{\tr}{tr}

\DeclareMathOperator{\Sym}{Sym}

\DeclareMathOperator{\Mod}{Mod}

\newcommand{\tors}{\mathrm{tors}}

\newcommand{\GL}{\mathbf{GL}}

\newcommand{\Sp}{\mathbf{Sp}}

\makeatletter
\@addtoreset{equation}{section}
\makeatother

\numberwithin{equation}{section}
\newtheorem{theorem}[equation]{Theorem}

\newtheorem{proposition}[equation]{Proposition}
\newtheorem{lemma}[equation]{Lemma}
\newtheorem{corollary}[equation]{Corollary}

\theoremstyle{definition}
\newtheorem{rmk}[equation]{Remark}

\newtheorem{eg}[equation]{Example}

\newtheorem{defn}[equation]{Definition}

\makeatletter
\renewcommand{\thesubsection}{%
  \ifnum\c@subsection<1 \@arabic\c@section
  \else \thesection.\@arabic\c@subsection
  \fi
}
\makeatother

\DeclareMathOperator{\chr}{char}

\author{Rob H.\ Eggermont}
\address{Department of Mathematics, University of Michigan, Ann Arbor, MI}
\email{\href{mailto:robegger@umich.edu}{robegger@umich.edu}}
\urladdr{\url{http://www-personal.umich.edu/~robegger/}}

\author{Andrew Snowden}
\address{Department of Mathematics, University of Michigan, Ann Arbor, MI}
\email{\href{mailto:asnowden@umich.edu}{asnowden@umich.edu}}
\urladdr{\url{http://www-personal.umich.edu/~asnowden/}}
\thanks{AS was supported by NSF grants DMS-1303082 and DMS-1453893 and a Sloan Fellowship.}

\title[Topological noetherianity for algebraic representations]{Topological noetherianity for algebraic representations of infinite rank classical groups}

\date{\today}

\begin{document}

\begin{abstract}
Draisma recently proved that polynomial representations of $\GL_{\infty}$ are topologically noetherian. We generalize this result to algebraic representations of infinite rank classical groups.
\end{abstract}

\maketitle

\section{Introduction}

\subsection{Background}

In recent years, a number of novel noetherianity results have been discovered and exploited; for but a few examples, see \cite{aschenbrenner, cohen, cohen2, fimodule, plucker, derksen, DraismaKuttler, eggermont, hillar, sym2noeth, putman-sam, delta, catgb}. It is not yet clear where the ultimate line between noetherianity and non-noetherianity lies: of the reasonable structures to consider, some are known to be noetherian and some not, and in between is a vast unknown territory.

Recently, Draisma \cite{draisma} proved a breakthrough result claiming a large tract of the unknown for the noetherian side. To state it, we must recall some terminology. Fix an algebraically closed field $\bk$ and let $\GL=\bigcup_{n \ge 1} \GL_n(\bk)$. A representation of $\GL$ is called {\bf polynomial} if it is a subquotient of a finite\footnote{In other settings, one allows infinite sums of tensor powers, but in this paper we restrict to finite sums.} direct sum of tensor products of the standard representation $\bV=\bigcup_{n \ge 1} \bk^n$. If $\bE$ is a polynomial representation then the dual space $\hat{\bE}$ is canonically identified with the $\bk$-points of the spectrum of the ring $\Sym(\bE)$, and in this way inherits a Zariski topology. Recall that if a group $G$ acts on a space $X$ (such as $\GL$ on $\hat{\bE}$) then we say that $X$ is {\bf topologically $G$-noetherian} if every descending chain of $G$-stable closed subsets of $X$ stabilizes. We can now state Draisma's theorem:

\begin{theorem}[Draisma]
Let $\bE$ be a polynomial representation of $\GL$. Then $\hat{\bE}$ is topologically $\GL$-noetherian.
\end{theorem}

This result has already found application: Erman--Sam--Snowden \cite{ess} have combined it with the resolution of Stillman's conjecture by Ananyan--Hochster \cite{hochster} to establish a vast generalization of Stillman's conjecture.

\subsection{The main theorem}

Draisma's theorem elicits a few natural follow-up questions. Does noetherianity hold for any non-polynomial representations of $\GL$? And what about representations of similar groups, such as the infinite orthogonal group? The purpose of this paper is to provide some answers to these questions.

Let $\bO=\bigcup_{n \ge 1} \bO_n(\bk)$ be the infinite orthogonal group, and let $\Sp=\bigcup_{n \ge 1} \Sp_{2n}(\bk)$ be the infinite symplectic group. We say that a representation of $\bO$ or $\Sp$ is {\bf algebraic} if it appears as a subquotient of a finite direct sum of tensor powers of the standard representation $\bV$. Let $\bV_*=\bigcup_{n \ge 1} (\bk^n)^*$, the so-called restricted dual of $\bV$. The group $\GL$ acts on $\bV_*$. A representation of $\GL$ is called {\bf algebraic} if it appears as a subquotient of a finite direct sum of representations of the form $\bV^{\otimes n} \otimes \bV_*^{\otimes m}$. We note that $\bV \cong \bV_*$ as representations of $\bO$ and $\Sp$, so one does not get a larger category of representations by using $\bV_*$. In characteristic~0, algebraic representations of $\GL$, $\bO$, and $\Sp$ have been studied in \cite{koszulcategory, penkovserganova, penkovstyrkas, infrank}.

The main theorem of this paper is:

\begin{theorem} \label{thm:main}
Let $\bG \in \{\bO, \Sp, \GL \}$ and let $\bE$ be an algebraic representation of $\bG$. Then $\hat{\bE}$ is topologically $\bG$-noetherian.
\end{theorem}

In \S \ref{s:counter}, we show that certain very similar looking statements are false. It would be interesting to see if this theorem has any applications in commutative algebra along the lines of \cite{ess}. We note, however, that \cite{ess} makes use of the kind of statements found in \S \ref{s:counter} that are true in the case of polynomial representations, but false for algebraic representations.

\subsection{Possible generalizations}

There are two directions in which Theorem~\ref{thm:main} might be generalized. First, one might hope for a statement at the level of ideals. That is, suppose that $\bE$ is an algebraic representation of $\bG$ (for $\bG$ as in the theorem), and let $A=\Sym(\bE)$. The theorem is equivalent to the statement that any ascending chain of $\bG$-stable radical ideals in $A$ stabilizes. A plausible stronger statement is that every ascending chain of $\bG$-stable ideals stabilizes. More generally, if $\bE'$ is a second algebraic representation, one might hope that every ascending chain of $\bG$-stable submodules of $A \otimes \bE'$ stabilizes.

Second, one might hope to extend topological noetherianity to certain analogs of algebraic representations. Let $A=\Sym(\Sym^2(\bC^{\infty}))$, equipped with its natural action of $\GL$. One can then consider the category $\Mod_A$ of $A$-modules equipped with a compatible polynomial action of $\GL$. Let $\Mod_A^{\tors}$ be the Serre subcategory of torsion modules. It is known \cite{sym2noeth} that the quotient category $\Mod_A/\Mod_A^{\tors}$ is equivalent to the category of algebraic representations of $\bO$ (as a tensor category). Thus Theorem~\ref{thm:main} in the case $\bG=\bO$ can be stated as: every finitely generated algebra in the category $\Mod_A/\Mod_A^{\tors}$ is topologically noetherian. It seems plausible that this statement might hold true for any $A$ of the form $\Sym(\bE)$, with $\bE$ a polynomial representation of $\GL$.

\subsection{Outline}

In \S \ref{s:red}, we show that it suffices to prove Theorem~\ref{thm:main} for any one of the three groups. In \S \ref{s:prelim}, we go over some preliminary material. The main theorem (in the general linear case) is proved in \S \ref{s:main}. In \S \ref{s:counter}, we discuss counterexamples to certain variants of the main theorem.

\subsection*{Acknowledgments}

We thank Steven Sam for helpful discussions.

\section{Reduction to the general linear case} \label{s:red}

We now show that it suffices to prove Theorem~\ref{thm:main} in any one of the three cases. Later, we will prove the theorem in the general linear case.

\begin{lemma} \label{lem:red}
Suppose that $\bG, \bH \in \{\GL,\Sp,\bO\}$ and that $\bH \to \bG$ is a homomorphism such that the standard representation of $\bG$ pulls back to an algebraic representation of $\bH$. Suppose also that Theorem~\ref{thm:main} holds for $\bH$. Then Theorem~\ref{thm:main} holds for $\bG$.
\end{lemma}

\begin{proof}
The hypothesis implies that any algebraic representation of $G$ pulls back to an algebraic representation of $\bH$. Let $\bE$ be an algebraic representation of $\bG$. Suppose that $Z_{\bullet} \subset \hat{\bE}$ is a descending chain of $\bG$-stable closed subset. Then it is also a descending chain of $\bH$-stable closed subsets, and thus stabilizes. Thus $\hat{\bE}$ is topologically $\bG$-noetherian.
\end{proof}

\begin{proposition}
If Theorem~\ref{thm:main} holds for any one of the three groups then it holds for the other two.
\end{proposition}

\begin{proof}
Suppose the theorem holds for $\GL$ and $\chr(\bk) \ne 2$. Consider the representation $\bV \oplus \bV_*$ of $\GL$. This representation carries both a symmetric form and a symplectic form, defined by the formulas $(v+\lambda,v'+\lambda')=\lambda(v') \pm \lambda'(v)$. These forms give group homomorphisms $\GL \to \bO$ and $\GL \to \Sp$. These homomorphisms have the property that the standard representation pulls back to the representation $\bV \oplus \bV_*$. Thus Theorem~\ref{thm:main} holds for $\bO$ and $\Sp$ by Lemma~\ref{lem:red}.

If $\chr(\bk)=2$ then a similar argument works. The bilinear form defined in the previous paragraph is alternating in characteristic~2 (i.e., it satisfies $(v+\lambda,v+\lambda)=0$) and thus yields a map $\GL \to \Sp$ that again allows us to apply Lemma~\ref{lem:red}. The representation $\bV \oplus \bV_*$ also admits a quadratic form defined by $(v,\lambda) \mapsto \lambda(v)$, which affords a homomorphim $\GL \to \bO$ to which we can apply Lemma~\ref{lem:red}.

Now suppose that Theorem~\ref{thm:main} holds for $\bO$. Applying Lemma~\ref{lem:red} to the inclusion $\bO \to \GL$, we see that Theorem~\ref{thm:main} holds for $\GL$. Appealing to the previous two paragraphs, we thus see that Theorem~\ref{thm:main} holds for $\Sp$ as well. Similarly, if Theorem~\ref{thm:main} holds for $\Sp$ then we get it for $\GL$ and then $\bO$.
\end{proof}

A similar argument is used in the following proposition, which we also require. A representation of $\GL \times \GL$ is {\bf polynomial} if it appears as a subquotient of a finite sum of representations of the form $\bV^{\otimes n} \otimes \bV^{\otimes m}$.

\begin{proposition} \label{prop:glgl}
Let $\bE$ be a polynomial representation of $\GL \times \GL$. Then $\hat{\bE}$ is topologically $\GL \times \GL$ noetherian.
\end{proposition}

\begin{proof}
Consider the diagonal copy of $\GL$ in $\GL \times \GL$. The restriction $\bE \vert_{\GL}$ is then polynomial: indeed, if $\bE$ is a subquotient of $\bigoplus_{i=1}^k \bV^{\otimes n_i} \otimes \bV^{\otimes m_i}$ then $\bE \vert_{\GL}$ is a subquotient of $\bigoplus_{i=1}^k \bV^{\otimes (n_i+m_i)}$. Since $\hat{\bE} \vert_{\GL}$ is topologically noetherian by Draisma's theorem, the result follows.
\end{proof}

\section{Preliminaries} \label{s:prelim}

\subsection{Spaces of matrices}

Let $\hat{\bM}$ be the set of matrices $(a_{i,j})_{i,j \in \bN}$ with $a_{i,j} \in \bk$. Let $\bM \subset \hat{\bM}$ be the subset consisting of matrices with only finitely many non-zero entries. We have a trace pairing
\begin{displaymath}
\langle, \rangle \colon \hat{\bM} \times \bM \to \bk, \qquad \langle A, B \rangle = \tr(A {}^tB)
\end{displaymath}
that identifies $\hat{\bM}$ with the linear dual of the space $\bM$. We let $\bM_n \subset \bM$ be the set of matrices $(a_{i,j})$ with $a_{i,j}=0$ for $i>n$ or $j>n$. We define $\hat{\bM}_n$ similarly, but regard it as a quotient of $\hat{\bM}$. Thus $\bM$ is the union of the $\bM_n$ and $\hat{\bM}$ is the inverse limit of the $\hat{\bM}_n$.

Let $\bU, \bL \subset \bM$ and $\hat{\bU}, \hat{\bL} \subset \hat{\bM}$ be the spaces of upper-triangular and lower-triangular matrices. The trace pairing identifies $\hat{\bL}$ with the linear dual of $\bL$ and $\hat{\bU}$ with the dual of $\bU$. One cannot multiply arbitrary elements of $\hat{\bM}$, as this would typically involve infinite sums. However, the product $AB$ is defined for $A \in \hat{\bL}$ and $B \in \hat{\bM}$, or for $A \in \hat{\bM}$ and $B \in \hat{\bU}$. We define $\bL_n$, $\bU_n$, $\hat{\bL}_n$, and $\hat{\bU}_n$ in the obvious ways.

We let $\bV_n=\bk^n$ and $\bV=\bigcup_{n \ge 1} \bV_n$. We let $\hat{\bV}$ and $\hat{\bV}_n$ be the dual spaces to $\bV$ and $\bV_n$, so that $\hat{\bV}$ is the inverse limit of the spaces $\hat{\bV}_n$.

\subsection{Polynomial representations}

Let $\bE$ be a polynomial representation of $\GL$. Then the action of $\GL_n \subset \GL$ extends uniquely to an action of the monoid $\bM_n$, and these assemble to an action of $\bM$. Furthermore, the action of $\bU$ on the dual $\hat{\bE}$ extends uniquely to a continuous action of $\hat{\bU}$, where continuous means that for fixed $v \in \hat{\bE}$ the quantity $uv$ depends only on the projection of $u$ to $\hat{\bU}_n$, for some $n$ depending only on $v$. This is easy to see when $\bE=\bV$: the point is that, if $e_i$ denotes the $i$th basis vector of $\bV$, then $u e_i^*$ only depends on the top left $i \times i$ block of $u$. The action of $\bL$ on $\hat{\bE}$ does \emph{not} similarly extend to an action of $\hat{\bL}$.

We can then equivalently think of the representation $\bE$ as a polynomial functor $\ul{\bE}$ on the category of $\bk$-vector spaces. We let $\bE_n$ be the value of the functor $\ul{\bE}$ on $\bk^n$, so that $\bE$ itself is identified with the union of the $\bE_n$. We let $\hat{\bE}$ and $\hat{\bE}_n$ be the linear duals of $\bE$ and $\bE_n$, so that $\hat{\bE}$ is the inverse limit of the $\hat{\bE}_n$. We note that $\bE_n \subset \bE$ is stable under the action of $\bU$ and $\hat{\bU}$. It follows that the projection map $\hat{\bE} \to \hat{\bE}_n$ is compatible with the action of $\hat{\bU}$.

For the purposes of this paper, we really only need to consider $\bE$'s that are finite sums of tensor powers of $\bV$. In this case, one does not need to think about polynomial functors: the space $\bE_n$ is then the corresponding sum of tensor powers of $\bV_n$.

We identify $\hat{\bE}$ with the spectrum of the ring $\Sym(\bE)$ (or more accurately, the $\bk$-points of the spectrum), and equip it with the Zariski topology. By definition, a regular function $f$ on $\hat{\bE}$ is an element of $\Sym(\bE)$. Since $\bE$ is the union of the $\bE_n$, it follows that $f$ belongs to $\Sym(\bE_n)$ for some $n$; we say that $f$ has {\bf level} $n$. This implies that $f$ factors through the projection $\hat{\bE} \to \hat{\bE}_n$.

\section{The general linear case} \label{s:main}

Let $\bG=\GL \times \GL$, and let $\bH$ be the subgroup of elements of the form $(g, {}^tg^{-1})$. Of course, $\bH$ is isomorphic to $\GL$. By definition, every algebraic representation of $\bH$ is a subquotient of $\bE \vert_{\bH}$ for some polynomial representation $\bE$ of $\bG$, so it suffices to prove noetherianity of $\hat{\bE} \vert_{\bH}$ for all such $\bE$. We therefore fix $\bE$ for the rest of this section, and prove that $\hat{\bE} \vert_{\bH}$ is noetherian.

Let $\bG$ act on $\bM$ by the formula $(g,h) \cdot A=gA {}^th$. The dual action of $\bG$ on $\hat{\bM}$ is given by the formula $(g,h) \cdot A = {}^tg^{-1} A h^{-1}$. Let $I \in \hat{\bM}$ be the identity matrix. Note that the stabilizer of $I$ in $\bG$ is exactly $\bH$. Given an $\bH$-stable closed subset $Z$ of $\hat{\bE}$, let $Z^+$ be the closure of the $\bG$-orbit of the set $\{I\} \times Z$ in $\hat{\bM} \times \hat{\bE}$. (The topology on the product is the Zariski topology on the product of schemes, which is not the product topology.) We will prove:

\begin{proposition} \label{prop:main}
Let $Z$ be an $\bH$-stable closed subset of $\hat{\bE}$. Then $\{I\} \times Z=(\{I\} \times \hat{\bE}) \cap Z^+$.
\end{proposition}

Before proving the proposition, we note an important consequence.

\begin{corollary}
The function $Z \mapsto Z^+$ defines an order-preserving injection
\begin{equation} \label{eq:map}
\{ \text{$\bH$-stable closed subsets of $\hat{\bE}$} \} \to
\{ \text{$\bG$-stable closed subsets of $\hat{\bM} \times \hat{\bE}$} \}
\end{equation}
In particular, $\hat{\bE}$ is topologically $\bH$-noetherian.
\end{corollary}

\begin{proof}
It is clear that $Z \mapsto Z^+$ is order preserving. By the proposition, we can recover $Z$ from $Z^+$, and so the map is injective. Since $\hat{\bM} \times \hat{\bE}$ is the dual of the polynomial representation $\bM \oplus \bE$ of $\bG$, Proposition~\ref{prop:glgl} shows that it is topologically $\bG$-noetherian. In particular, the right side of \eqref{eq:map} satisfies the descending chain condition. It follows that the left side does as well, and so $\hat{\bE}$ is topologically $\bH$-noetherian.
\end{proof}

Let $\hat{\bU}'$ be the subset of $\hat{\bU}$ consisting of matrices where all diagonal entries are~1, and let $\hat{\bW}=\hat{\bU} \times \hat{\bU}'$. Let $\varphi \colon \hat{\bW} \to \hat{\bM}$ be the function defined by $\phi(u,v)={}^tuIv$. Note that since ${}^tu \in \hat{\bL}$ and $v \in \hat{\bU}$, this matrix product is defined. Let $\varphi_n \colon \hat{\bW}_n \to \hat{\bM}_n$ be defined by the same formula. We note that the diagram
\begin{displaymath}
\xymatrix{
\hat{\bW} \ar[r]^{\varphi} \ar[d] & \hat{\bM} \ar[d] \\
\hat{\bW}_n \ar[r]^{\varphi_n} & \hat{\bM}_n }
\end{displaymath}
commutes. That is, if $i,j \le n$ then the $(i,j)$ entry of $\varphi(u,v)$ only depends on the top left $n \times n$ blocks of $u$ and $v$.

\begin{lemma} \label{lem:1}
Let $X$ be an affine variety over $\bk$, and let $h \colon \hat{\bW}_n \times X \to \bk$ be a regular function. Then there is a monomial $m$ in the diagonal entries of $\hat{\bU}_n$ and a regular function $H \colon \hat{\bM}_n \times X \to \bk$ such that $H(\varphi(w), x)=m(u) h(w,x)$ holds for all $w=(u,v) \in \hat{\bW}_n$ and $x \in X$.
\end{lemma}

\begin{proof}
Consider the polynomial ring $R = \Sym(\bW_n) = \bk[u_{ij}, v_{kl}]_{1 \leq i \leq j \leq n, 1 \leq k < l \leq n} \otimes \cO_X$. The morphism $\phi_n$ induces an injective ring homomorphism $\Sym(\bM_n) \otimes \cO_X \to R$, so we may view the former as a subring of the latter. Note that this subring is generated (as an $\cO_X$-algebra) by terms of the form $\sum_{k=1}^iu_{ki}v_{kj}$ for $i < j$ and terms of the form $u_{ji} + \sum_{k=1}^{j-1}u_{ki}v_{kj}$ for $i \geq j$. We denote these terms by $a_{ij}$. It thus suffices to show that all $u_{ij}$ and $v_{kl}$ can be expressed as a quotient of a polynomial in the $a_{ij}$ by a monomial in the $u_{ii}$. We do so by induction.

We have $u_{1i} = a_{i1}$ for any $i \geq 1$ and $v_{1j} = \frac{a_{1j}}{u_{11}}$ for any $j > 1$. Both of these are expressions of the desired form. Now suppose that $u_{ki}$ and $v_{kj}$ can be expressed in the desired form for all $k \leq i < j$ with $k < K$, for some $K$. For $i \geq K$ we see that $a_{iK}$ is equal to $u_{Ki}$ plus terms of the form $u_{ki}v_{kK}$ with $k < K$; since we have an expression of the desired form for each of these other terms, we find one for $u_{Ki}$. For $j > K$, we find that $a_{Kj}$ is equal to $u_{KK}v_{Kj}$ plus terms of the form $u_{kK}v_{kj}$ with $k < K$; once again, this yields an expression of the desired form for $v_{Kj}$. This concludes the proof of the lemma.
\end{proof}

Let $f$ be a regular function on $\hat{\bE}$ of level $n$ that vanishes on $Z$. We define a function $h_f$ on $\hat{\bW} \times \hat{\bE}$ by
\begin{displaymath}
h_f(w, x)= f(w \cdot x).
\end{displaymath}
Note that the above formula makes use of the action map $\hat{\bU} \times \hat{\bU} \times \hat{\bE} \to \hat{\bE}$, which exists since $\hat{\bE}$ is a polynomial representation of $\bG$.

We claim that $h_f$ is a regular function of level $n$. Consider the diagram
\begin{displaymath}
\xymatrix{
\hat{\bW} \times \hat{\bE} \ar[r] \ar[d] & \hat{\bE} \ar[r]^f \ar[d] & \bk \ar@{=}[d] \\
\hat{\bW}_n \times \hat{\bE}_n \ar[r] & \hat{\bE}_n \ar[r]^f & \bk }
\end{displaymath}
The left horizontal maps are the action maps. Both squares commute, so the whole diagram does, and thus $h_f$ factors through the left map and so has level $n$.

It follows from the Lemma~\ref{lem:1} that there is a monomial $m_f$ in the diagonal entries of $\hat{\bU}_n$ such that $m_f h_f$ induces a regular function on $\hat{\bM} \times \hat{\bE}$ of level $n$; call this function $H_f$.

\begin{lemma} \label{lem:2}
The function $H_f$ vanishes on $Z^+$.
\end{lemma}

\begin{proof}
Let $\bG_m=\GL_m \times \GL_m$, and let $\bH_m \subset \bG_m$ be the set of matrices of the form $(g, {}^tg^{-1})$. It suffices to show that $H_f$ vanishes on a dense subset of the $\bG_m$-orbit of $\{I\} \times Z$ for all $m \gg 0$. Now, the set $\bW_m \bH_m \subset \bG_m$ is dense, since a generic element of $\GL_m$ can be written as a product of an upper triangular and lower triangular matrix. Suppose that $g=wh$, with $m \ge n$. Writing $w=(u,v)$, we have $g=(uh,v{}^th^{-1})$. Note that
\begin{displaymath}
g \cdot I={}^t(uh)^{-1} I (v{}^th^{-1})^{-1}={}^tu^{-1} I v^{-1}=\varphi(w^{-1}).
\end{displaymath}
Let $z \in Z$. We then have
\begin{displaymath}
H_f(g(I,z))=H_f(\varphi(w^{-1}), gz)=m_f(u^{-1}) h_f(w^{-1}, gz).
\end{displaymath}
Now, by definition, we have
\begin{displaymath}
h_f(w^{-1}, gz)=f(w^{-1} whz)=f(hz)
\end{displaymath}
Since $Z$ is $\bH$-stable, we have $hz \in Z$. Since $f$ vanishes on $Z$, we thus see that this expression vanishes.
\end{proof}

\begin{lemma} \label{lem:3}
Suppose $z \in \hat{\bE}$ does not belong to $Z$. Then there exists $f$ vanishing on $Z$ such that $H_f(I,z) \ne 0$.
\end{lemma}

\begin{proof}
Let $n$ be such that the projection $\ol{z}$ of $z$ to $\hat{\bE}_n$ does not belong to $Z_n$. We can thus find a regular function $f$ on $\hat{\bE}_n$ that vanishes on $Z_n$ but does not vanish at $\ol{z}$. We have
\begin{displaymath}
H_f(I,z)=H_f(\varphi(I,I), z)=m_f(I) h_f((I,I),z).
\end{displaymath}
Since $m_f(u)$ is monomial in the diagonal entries of $u$, it follows that $m_f(I) \ne 0$. Furthermore, we have $h_f(I,I,z)=f(z) \ne 0$. We thus see that $H_f(I,z) \ne 0$.
\end{proof}

\begin{proof}[Proof of Proposition~\ref{prop:main}]
It is clear that $\{I\} \times Z$ is contained in the intersection of $\{I\} \times \hat{\bE}$ and $Z^+$. Suppose that $z \in \hat{\bE}$ does not belong to $Z$. Let $f$ be as in Lemma~\ref{lem:3}. By Lemma~\ref{lem:2}, we see that $H_f$ belongs to the ideal of $Z^+$. Since $H_f(I,z) \ne 0$, it follows that $(I,z) \not\in Z^+$, which proves the result.
\end{proof}

\section{Some counterexamples} \label{s:counter}

Draisma's theorem states that if $\bE$ is a polynomial representation of $\bG$ then $\hat{\bE}$ is topologically $\bG$-noetherian. One can identify $\bE_n$ with the spectrum of $\Sym(\hat{\bE}_n)$, and in this way regard $\bE=\varinjlim \bE_n$ as an ind-scheme. (The elements of $\Sym(\hat{\bE})$ define functions on $\bE$, and their zero loci define the closed sets.) It is therefore sensible to ask if $\bE$ is topologically noetherian. In \cite{ess} it is proved that this is the case: in fact, the $\bG$-stable closed sets of $\bE$ and $\hat{\bE}$ are shown to be in bijection, and so noetherianity of $\bE$ follows from Draisma's theorem.

We now explain that ind- version of Theorem~\ref{thm:main} fails. In fact, noetherianity fails due to an obvious obstruction in each case.

First consider the $\GL$ case. If $m$ is an element of $\bM$ then one can make sense of the determinant $\chi(m)=\det(1-tm)$, a polynomial in $t$. The action of $\GL$ on $\bM$ by conjugation defines an algebraic representation, and leaves $\chi$ invariant. Let $c_i(m)$ be the coefficient of $t^i$ in $\chi(m)$. Then $c_i \colon \bM \to \bA^1$ is a $\GL$-invariant function. Taken together, the $c_i$ define an $\GL$-invariant function $c \colon \bM \to \bA^{\infty}$, where the target is the ind-scheme $\varinjlim \bA^n$. It is easy to see that $c$ is surjective, from which it follows that $\bM$ is not topologically $\bH$-noetherian. (If $Z_{\bullet}$ is an infinite strictly descending chain of closed subsets of $\bA^{\infty}$ then $c^{-1}(Z_{\bullet})$ is an infinite strictly descending chain of $\bH$-stable closed subsets of $\bS$.)

The other cases are similar. In the symplectic case, one considers the characteristic polynomial of anti-symmetric matrices, while in the orthogonal case one uses symmetric matrices. We remark that $\hat{\bV}$ (and more generally $\hat{\bV}^d$ for any $d \ge 1$) is known to be topologically $\fS$-noetherian \cite{cohen,cohen2,aschenbrenner,hillar}, where $\fS$ is the infinite symmetric group, and for similar reasons $\bV$ is not topologically $\fS$-noetherian (one makes an invariant polynomial by using the coordinates as roots, so that the $c_i$'s are elementary symmetric functions).

It is interesting to observe that these counterexamples do not apply in the pro- setting since the invariants they use no longer make sense: one cannot take the characteristic polynomial of an element of $\hat{\bM}$, as this would involve infinite sums. Similarly, the proof of noetherianity in the pro- setting does not apply in the ind-setting since the element $I \in \hat{\bM}$ does not belong to $\bM$.


\begin{thebibliography}{[CEF]}

\bibitem[AH]{aschenbrenner}
Matthias Aschenbrenner, Christopher Hillar. Finite generation of symmetric ideals. {\it Trans.\ Amer.\ Math.\ Soc.} {\bf 359} (2007), 5171--5192. \arxiv{math/0411514}.

\bibitem[AH2]{hochster}
Tigran Ananyan, Melvin Hochster. Small Subalgebras of Polynomial Rings and Stillman's Conjecture. \arxiv{1610.09268}

\bibitem[Co]{cohen}
D.~E.~Cohen. On the laws of a metabelian variety. {\it J. Algebra} {\bf 5} (1967), 267--273.

\bibitem[Co2]{cohen2}
D.~E.~Cohen. Closure relations, Buchberger's algorithm, and polynomials in infinitely many variables. Computation theory and logic, (1987), 78--87.

\bibitem[CEF]{fimodule}
Thomas Church, Jordan S. Ellenberg, Benson Farb. FI-modules and stability for representations of symmetric groups. {\it Duke Math. J.} {\bf 164} (2015), no.~9, 1833--1910. \arxiv{1204.4533v4}

\bibitem[Dr]{draisma}
Jan Draisma. Topological Noetherianity of polynomial functors. \arxiv{1705.01419}

\bibitem[DE]{plucker}
Jan Draisma, Rob H. Eggermont. Pl\"ucker varieties and higher secants of Sato's Grassmannian. {\it J. Reine Angew. Math.}, to appear. \arxiv{1402.1667}

\bibitem[DES]{derksen}
Harm Derksen, Rob Eggermont, Andrew Snowden. Topological noetherianity for cubic polynomials. {\it Algebra \& Number Theory}, to appear. \arxiv{1701.01849}

\bibitem[DK]{DraismaKuttler}
J.~Draisma, J.~Kuttler. Bounded-rank tensors are defined in bounded degree. {\it Duke Math.\ J.}{\bf 163} (2014), no.~1, 35--63. \arxiv{1103.5336}

\bibitem[DPS]{koszulcategory}
Elizabeth Dan-Cohen, Ivan Penkov, Vera Serganova. A Koszul category of representations of finitary Lie algebras. {\it Adv.\ Math.} {\bf 289} (2016), 250–-278. \arxiv{1105.3407v2}

\bibitem[Eg]{eggermont}
Rob H. Eggermont. Finiteness properties of congruence classes of infinite matrices. {\it Linear Algebra Appl.} {\bf 484} (2015), 290--303. \arxiv{1411.0526}

\bibitem[ESS]{ess}
Daniel Erman, Steven Sam, Andrew Snowden. Generalizations of Stillman's conjecture via twisted commutative algebras. In preparation.

\bibitem[HS]{hillar}
Christopher Hillar, Seth Sullivant. Finite Gr\"obner bases in infinite dimensional polynomial rings and applications. {\it Adv.\ in Math.} {\bf 221} (2012), 1--25. \arxiv{0908.1777}

\bibitem[NSS]{sym2noeth} Rohit Nagpal, Steven V Sam, Andrew Snowden. Noetherianity of some degree two twisted commutative algebras. {\it Selecta Math. (N.S.)} {\bf 22} (2016), no.~2, 913--937. \arxiv{1501.06925v2}

\bibitem[PSa]{putman-sam}
Andrew Putman, Steven~V Sam. Representation stability and finite linear groups. {\it Duke Math. J.}, to appear. \arxiv{1408.3694v2}

\bibitem[PSe]{penkovserganova}
Ivan Penkov, Vera Serganova. Categories of integrable $sl(\infty)$-, $o(\infty)$-, $sp(\infty)$-modules. {\it Representation Theory and Mathematical Physics}, Contemp. Math. {\bf 557}, AMS 2011, pp. 335--357. \arxiv{1006.2749v1}

\bibitem[PSt]{penkovstyrkas}
Ivan Penkov, Konstantin Styrkas. Tensor representations of classical locally finite Lie algebras. {\it Developments and trends in infinite-dimensional Lie theory}, Progr. Math. {\bf 288}, Birkh\"auser Boston, Inc., Boston, MA, 2011, pp. 127--150. \arxiv{0709.1525v1}

\bibitem[Sn]{delta}
Andrew Snowden. Syzygies of Segre embeddings and $\Delta$-modules. {\it Duke Math.\ J.} {\bf 162} (2013), no.~2, 225--277. \arxiv{1006.5248v4}

\bibitem[SS]{infrank}
Steven~V Sam, Andrew Snowden. Stability patterns in representation theory. {\it Forum Math. Sigma} {\bf 3} (2015), e11, 108 pp. \arxiv{1302.5859v2}.

\bibitem[SS2]{catgb}
Steven~V Sam, Andrew Snowden. Gr\"obner methods for representations of combinatorial categories. {\it J. Amer. Math. Soc.}, to appear. \arxiv{1409.1670v3}


\end{thebibliography}
\end{document}